\documentclass[12pt,reqno]{article}

\usepackage[usenames]{color}
\usepackage{amsmath,amssymb,amsthm}
\usepackage{url}

\usepackage[colorlinks=true,
linkcolor=webgreen,
filecolor=webbrown,
citecolor=webgreen]{hyperref}

\definecolor{webgreen}{rgb}{0,.5,0}
\definecolor{webbrown}{rgb}{.6,0,0}

\theoremstyle{plain}
\newtheorem{theorem}{Theorem}

\theoremstyle{definition}

\theoremstyle{remark}
\newtheorem*{remark}{Remark}

\newcommand{\lcm}{\mathop{\mathrm{lcm}}}
\newcommand{\Disc}{\mathop{\mathrm{Disc}}}
\newcommand{\Res}{\mathop{\mathrm{Res}}}
\newcommand{\beq}{\begin{equation}}
\newcommand{\eeq}{\end{equation}}

\newcommand{\seqnum}[1]{\href{http://oeis.org/#1}{\underline{#1}}}

\begin{document}

\begin{center}

\vskip 1cm{\LARGE\bf 
On integral points on biquadratic curves and 
\vskip .1in
near-multiples of squares in Lucas sequences}

\vskip 1cm

\large

Max A. Alekseyev\\
Department of Mathematics \\
The George Washington University \\
2115 G St. NW, Washington, DC 20052 \\
United States \\
\href{mailto:maxal@gwu.edu}{\tt maxal@gwu.edu} \\
\ \\
Szabolcs Tengely \\
Mathematical Institute \\ 
University of Debrecen \\
4010 Debrecen, PO Box 12 \\
Hungary \\
\href{mailto:tengely@science.unideb.hu}{\tt tengely@science.unideb.hu} \\
\end{center}

\vskip .2 in

\begin{abstract}
We describe an algorithmic reduction of the search for integral points on a curve $y^2 = ax^4 + bx^2 + c$ 
with $ac(b^2-4ac)\ne 0$
to solving a finite number of Thue equations.
While existence of such reduction is anticipated 
from arguments of algebraic number theory, 
our algorithm is elementary and to best of our knowledge 
is the first published algorithm of this kind.
In combination with other methods and powered by existing software Thue equations solvers,
it allows one to efficiently compute integral points on biquadratic curves.

We illustrate this approach with a particular application of finding near-multiples of squares in Lucas sequences.
As an example, we establish that among Fibonacci numbers only $2$ and $34$ 
are of the form $2m^2+2$; only $1$, $13$, and $1597$ are of the form $m^2-3$; and so on.

As an auxiliary result, we also give an algorithm for solving a Diophantine equation $k^2 = \tfrac{f(m,n)}{g(m,n)}$ in integers $m, n, k$, 
where $f$ and $g$ are homogeneous quadratic polynomials.
\end{abstract}

\section{Introduction}

Siegel~\cite{Siegel1929} proved that any equation $y^2=f(x)$ with irreducible polynomial $f\in\mathbb{Z}[x]$ 
of degree at least 3 has finitely many integral points. With the method of Baker~\cite{Baker1966,Baker1967a,Baker1967b}, it became possible to 
bound the solutions and perform an exhaustive search. For third-degree curves, Baker's method was a subject to many practical improvements,
and now there exists a number of software implementations for finding integral points on elliptic curves~\cite{MAGMA,eclib,SAGE}.
These procedures are based on a method developed by Stroeker and Tzanakis \cite{StTz} and independently by Gebel, Peth\H{o} and Zimmer \cite{GPZell}.

Thue equations of the form $g(x,y)=d$, where $g\in\mathbb{Z}[x,y]$ is a homogeneous irreducible polynomial of degree at least 3 and $d\in\mathbb{Z}$,
were first studied by Thue~\cite{Thue1909}, who proved that they have only a finite number of integer solutions.
In computer era, Thue equations became a subject to developments of computational methods,
resulting in at least two implementations: in computer algebra systems MAGMA~\cite{MAGMA} and PARI/GP~\cite{PARI}.
For our practical computations, we chose SAGE~\cite{SAGE}, which adopts PARI/GP Thue equations solver based on Bilu and Hanrot's improvement~\cite{Bilu1996}
of Tzanakis and de Weger's method~\cite{Tzanakis1989}.

In the current work, we show how to reduce a search for integral points on a biquadratic curve 
$$y^2 = ax^4 + bx^2 + c$$ 
with integer (or, more generally, rational) coefficients $a$, $b$, $c$ with $ac(b^2-4ac)\ne 0$ firstly to a Diophantine equation
$k^2 = \tfrac{f(m,n)}{g(m,n)}$ 
in coprime integers $m$, $n$ with homogeneous quadratic polynomials $f$ and $g$ (Theorem~\ref{TBiqCur}),
and then to a finite number of quartic Thue equations (Theorem~\ref{TFracSq}).\footnote{In the course of this reduction, 
we may also encounter Thue-like equations $g(x,y)=d$ with $g$ being a \emph{reducible} homogeneous polynomial of degree 4.
Such equations however are easily solvable~\cite[Theorem~3]{Alekseyev2011}.}
While possibility of reduction to Thue equations was described by Mordell~\cite{Mordell1969} based on arguments from algebraic number theory,
to the best of our knowledge, there is no published algorithm applicable for the general case.
Furthermore, in contrast to traditional treatment of this kind of problems with algebraic number theory~\cite{Mordell1969,Steiner1991,Weger1995,Petho1998},
our reduction method is elementary.
It may be viewed as a generalization of the method of Steiner and Tzanakis~\cite{Steiner1991}
who reduced Ljunggren equation $y^2 = 2x^4 - 1$ to two Thue equations (in particular, we obtain the same Thue equations).

There are other methods which can be used in certain cases to determine all integral solutions of the equation $y^2 = ax^4 + bx^2 + c$.
Poulakis~\cite{RungeP} provided an elementary algorithm to solve Diophantine equations of the form $y^2=f(x)$,
where $f(x)$ is quartic monic polynomial with integer coefficients,
based on Runge's method~\cite{RungeR, RungeS, RungeW}. Here it is crucial
that the leading coefficient of $f(x)$ is 1 (the idea also works if the leading coefficient is a perfect square). 
Using the theory of Pell equations, Kedlaya~\cite{Kedlaya-Pell} described a method to solve the system of equations
$$
\begin{cases}
x^2-a_1y^2 = b_1,\\
P(x,y) = z^2,
\end{cases}
$$
where $P$ is a given integer polynomial, and implemented his algorithm in Mathematica.\footnote{Kedlaya's implementation is available from \url{http://math.ucsd.edu/~kedlaya/papers/pell.tar}.} 
If we set $P(x,y)=c_1x+d_1,$ then we obtain a quartic equation of the form
$(a_1c_1y)^2=a_1z^4-2a_1d_1z^2-a_1b_1c_1^2.$
There is also a simple reduction of the equation $y^2 = ax^4 + bx^2 + c$ to an elliptic equation: after multiplying the equation by $a^2x^2$, one 
obtains
$$
(axy)^2=(ax^2)^3+b(ax^2)^2+ac(ax^2),
$$
which can be further written as
$$
Y^2=X^3+bX^2+acX.
$$
As we noted earlier to determine all integral points on a given elliptic curve one can follow a method developed
by Stroeker and Tzanakis \cite{StTz} and independently by Gebel, Peth\H{o} and Zimmer \cite{GPZell}.
The disadvantage of this approach is that there is no known algorithm to determine the rank of the so-called Mordell-Weil group
of an elliptic curve, which is necessary to determine all integral points on the curve. 

For efficient computation of integral points on biquadratic curves, 
we implemented in SAGE a combination of the elliptic curve and reduction to Thue equations methods.\footnote{Our implementation 
is available from \url{http://www.math.unideb.hu/~tengely/biquadratic.sage}.}
By default we employ the elliptic curve method and if it fails, we fall back to our reduction to Thue equations.

Our approach also allows one to efficiently compute solutions to a system of Diophantine equations (Theorem~\ref{Tzz2}):
$$
\begin{cases} 
a_1 x^2 + c_1 z = d_1,\\
b_2 y^2 + c_2 z^2 = d_2.
\end{cases}
$$
From this perspective, it continues earlier work~\cite{Alekseyev2011}, where the first author described an algorithm for computing solutions to a system of Diophantine equations:
$$
\begin{cases} 
a_1 x^2 + b_1 y^2 + c_1 z^2 = d_1,\\
a_2 x^2 + b_2 y^2 + c_2 z^2 = d_2,
\end{cases}
$$
and demonstrated applications for finding common terms of distinct Lucas sequences of the form $U(P,\pm 1)$ or $V(P,\pm 1)$,
which include Fibonacci, Lucas, Pell, and Lucas-Pell numbers.
The current method also has applications for such Lucas sequences, allowing one to find all terms 
of the form $a\cdot m^2 + b$ for any fixed integers $a$, $b$.
While the question of finding multiples of squares (i.e., with $b=0$) in Lucas sequences has been widely studied, starting with the works of 
Cohn~\cite{Cohn1964} and Wyler~\cite{Wyler1964}
(we refer to Bremner and Tzanakis~\cite{Bremner2007a} for an extensive review of the literature),
finding \emph{near}-multiples of squares (i.e., with $b\ne 0$) has got so far only a limited attention~\cite{Finkelstein1973,Finkelstein1975,Robbins1981,WalshNear}.
In the current work, we present an unified computational approach for solving this problem.
As an example, we establish 
that among Fibonacci numbers 
only $2$ and $34$ are of the form $2m^2+2$; only $1$, $13$, and $1597$ are of the form $m^2-3$; and so on.

The paper is organized as follows. In Section~\ref{SecHomPol}, 
we develop our machinery for homogeneous quadratic polynomials with integer coefficients.
In Section~\ref{SecBiQ}, we prove our method for finding integral points on biquadratic curves
and illustrate its workflow on Ljunggren equation.
In Section~\ref{SecInLucas}, we further demonstrate how our method can be used for finding near-multiples of squares in Lucas sequences and 
list some results of this kind.

\section{Homogeneous quadratic polynomials}\label{SecHomPol}

We start with studying properties of quadratic homogeneous polynomials with integer coefficients in two and three variables.
We do not distinguish between homogeneous polynomials in two variables from their univariate counterparts 
(i.e., $f(x,y)=ax^2 + bxy + cy^2$ and $\tilde{f}(z)=az^2 + bz + c$) that allows us to define resultant ($\Res$) and discriminant ($\Disc$) on them.

\begin{theorem}[Theorem~5 in \cite{Alekseyev2011}\footnote{This theorem corrects an error in \cite[Corollary~6.3.8]{Cohen07}.}]
\label{ThABC}
Let $A, B, C$ be non-zero integers and let $(x_0, y_0, z_0)$ with $z_0 \ne 0$ be a particular non-trivial integer solution
to the Diophantine equation $Ax^2+By^2+Cz^2=0$. Then its general integer solution is given by
\begin{equation}\label{SolSys}
(x, y, z) = \frac{p}{q}\; ( P_x(m, n),\; P_y(m, n),\; P_z(m, n))
\end{equation}
where $m, n$ as well as $p, q$ are coprime integers with $q>0$ dividing $2 \lcm(A,B) Cz_0^2$, and
\begin{equation}\label{Pxyz}
\begin{array}{lcl}
P_x(m,n) & = & x_0 A m^2 + 2y_0 Bmn - x_0 B n^2,\\
P_y(m,n) & = & -y_0 A m^2 + 2x_0 A mn + y_0 B n^2,\\
P_z(m,n) & = & z_0 A m^2 + z_0 B n^2.
\end{array}
\end{equation}
\end{theorem}

We refer to \cite{Cohen07,Cremona03} for general methods of finding a particular solution to a quadratic homogeneous equation 
in three variables.\footnote{In PARI/GP, a particular solution can be computed with the function \emph{bnfisnorm}.}

\begin{theorem}\label{TPolGCD} Let $P_1(x,y)$ and $P_2(x,y)$ be homogeneous quadratic polynomials with integer coefficients
and $R=\Res(P_1,P_2)\ne 0$. 
Let $G$ be the largest element in the Smith normal form of the resultant matrix of $P_1$ and $P_2$.\footnote{While $G=R$ would also satisfy the theorem statement,
we want $G$ as small as possible. The resultant $R$ is often much larger than $G$ defined in the theorem.}
Then for any coprime integers $m,n$, $\gcd(P_1(m,n),P_2(m,n))$ divides $G$.
\end{theorem}

\begin{proof} 
Let $P_1(x,y)=a_1 x^2+b_1 x y + c_1 y^2$ and $P_2(x,y) = a_2 x^2+b_2 x y + c_2 y^2$ where $a_1,b_1,c_1,a_2,b_2,c_2$ are integer.
Consider polynomials $x\cdot P_1(x,y)$, $y\cdot P_1(x,y)$, $x\cdot P_2(x,y)$, $y\cdot P_2(x,y)$ as linear combinations with integer coefficients of
the basis terms $x^3$, $x^2y$, $xy^2$, $y^3$.
Our goal is to find two linear combinations of these polynomials: 
one equal an integer multiple of $x^3$ and the other equal an integer multiple of $y^3$. 
This corresponds to the following two systems of linear equations with the same resultant matrix:
\beq\label{EqTwoSys}
\begin{matrix}
x^3:\\
x^2y:\\
xy^2:\\
y^3:
\end{matrix}
\qquad
\begin{bmatrix}
a_1 & 0 & a_2 & 0 \\
b_1 & a_1 & b_2 & a_2 \\
c_1 & b_1 & c_2 & b_2 \\
 0  & c_1 & 0   & c_2
\end{bmatrix}
\cdot
\begin{bmatrix}
t_1\\
t_2\\
t_3\\
t_4
\end{bmatrix}
\quad 
=
\quad
\begin{bmatrix}
1\\
0\\
0\\
0
\end{bmatrix}
\quad
\text{or}
\quad
\begin{bmatrix}
0\\
0\\
0\\
1
\end{bmatrix}
\eeq
with respect to rational numbers $t_1,t_2,t_3,t_4$. 
Since the matrix determinant equals the resultant $R\ne 0$,
the systems have unique solutions of the form:
$$(t_1,t_2,t_3,t_4) = \frac{1}{R} \left( \Delta_1, \Delta_2, \Delta_3, \Delta_4 \right)
= \frac{1}{G} \left( \frac{\Delta_1}{d_3}, \frac{\Delta_2}{d_3}, \frac{\Delta_3}{d_3}, \frac{\Delta_4}{d_3} \right)$$
where each $\Delta_i$ is the determinant of a certain $3\times 3$ minor of the resultant matrix 
and $d_3$ is its third determinant divisor. Here we used the fact that $\tfrac{R}{d_3}=G$ is the fourth elementary divisor of the resultant matrix 
(and the largest element of its Smith normal form).
It is important to notice that all vector components $\tfrac{\Delta_i}{d_3}$ are integer.

So we have two linear combinations of $x\cdot P_1(x,y)$, $y\cdot P_1(x,y)$, $x\cdot P_2(x,y)$, $y\cdot P_2(x,y)$ with integer coefficients 
$\tfrac{\Delta_1}{d_3}$, $\tfrac{\Delta_2}{d_3}$, $\tfrac{\Delta_3}{d_3}$, $\tfrac{\Delta_4}{d_3}$ equal $G\cdot x^3$ and $G\cdot y^3$, respectively.
For $(x,y)=(m,n)$ where $m,n$ are coprime integers, these linear combinations imply that $\gcd(P_1(m,n),P_2(m,n))$ 
divides both $G\cdot m^3$ and $G\cdot n^3$. Therefore, $\gcd(P_1(m,n),P_2(m,n))$ divides $\gcd(G\cdot m^3,G\cdot n^3) = G\cdot \gcd(m^3,n^3) = G$.
\end{proof}

\begin{remark} To compute $G$ in practice, we do not need to compute Smith normal form of the resultant matrix. 
Instead, we simply solve the two linear systems \eqref{EqTwoSys} and define $G$ as the least common multiple of all denominators in both solutions $(t_1,t_2,t_3,t_4)$.
\end{remark}

\begin{theorem}\label{TPol2Sq} Any homogeneous quadratic polynomial with integer coefficients and non-zero discriminant can be represented 
as a linear combination with non-zero rational coefficients of squares of two homogeneous linear polynomials.
Moreover, these polynomials are linearly independent.
\end{theorem}

\begin{proof} Let $a x^2+b x y + c y^2$ be a homogeneous quadratic polynomial with integer coefficients and non-zero discriminant, i.e., $b^2-4ac\ne 0$.
If $b=0$ then $ac\ne 0$ and the statement is trivial.

Suppose that $b\ne 0$. If $a\ne 0$, then we have
$$a x^2+b x y + c y^2 = \frac{1}{4a}\cdot (2a x + by)^2 + \frac{4ac - b^2}{4a}\cdot y^2.$$
Similarly, if $c\ne 0$, then we have
$$a x^2+b x y + c y^2 = \frac{4ac - b^2}{4c}\cdot x^2 + \frac{1}{4c}\cdot (bx + 2cy)^2.$$
Finally, if $a=c=0$, then
$$bxy = \frac{b}{4}\cdot (x+y)^2 - \frac{b}{4}\cdot (x-y)^2.$$
It is easy to see that in all cases, the linear polynomials are linearly independent.
\end{proof}

\begin{theorem}\label{TFracSq} Let $P_1(x,y)$ and $P_2(x,y)$ be homogeneous quadratic polynomials
with integer coefficients such that $\Disc(P_1)\ne 0$, $\Disc(P_2)\ne 0$, and $\Res(P_1,P_2)\ne 0$.
Then the equation
\beq
\label{EqSqFrac}
z^2 = \frac{P_1(x,y)}{P_2(x,y)}
\eeq
has a finite number of integer solutions $(x,y,z)=(m,n,k)$ with $\gcd(m,n)=1$.
\end{theorem}

\begin{proof} 
Suppose that $(x,y,z)=(m,n,k)$ with $\gcd(m,n)=1$ satisfies the equation \eqref{EqSqFrac}.
Since $P_2(m,n)$ divides $P_1(m,n)$, we have $\gcd(P_1(m,n),P_2(m,n)) = P_2(m,n)$, which by Theorem~\ref{TPolGCD} must divide a certain integer $G$. 
Then for some divisor\footnote{Unless specified otherwise, the divisors of an integer include both positive and negative divisors.} 
$g$ of $G$, we have $P_2(m,n)=g$ and $P_1(m,n) = gk^2$.
So $(x,y,z)=(m,n,k)$ represents a solution to the following system of equations:
\begin{equation}\label{SysSplit}
\begin{cases}
P_1(x,y) = g\cdot z^2, \\
P_2(x,y) = g.
\end{cases}
\end{equation}
Therefore, to find all solutions to \eqref{EqSqFrac}, we need to solve the systems \eqref{SysSplit} for $g$ ranging over the divisors of $G$.
Each such system is solved as follows.

We start with 
using Theorem~\ref{TPol2Sq} to represent $P_1(x,y)$ 
as a linear combination with rational coefficients of squares of two linearly independent polynomials, say, $P_1(x,y) = a\cdot Q_1(x,y)^2 + b\cdot Q_2(x,y)^2$.
Substituting this representation into the first equation of \eqref{SysSplit}, we get the following equation:
\beq\label{EqQ1Q2}
a\cdot Q_1(x,y)^2 + b\cdot Q_2(x,y)^2 - g\cdot z^2 = 0.
\eeq
We solve this equation using Theorem~\ref{ThABC} to obtain $Q_1(x,y) = \tfrac{p}{q}\cdot R_1(m,n)$ 
and
$Q_2(x,y) = \tfrac{p}{q}\cdot R_2(m,n)$, 
where $q$ ranges over the positive divisors of a certain integer and integer $p$ is coprime to $q$.
We solve this system of linear equations with respect to $x,y$ to obtain
$$(x,y) = \frac{p}{q}\cdot \left(S_x(m,n),\; S_y(m,n) \right),$$ 
where $S_x(m,n)$ and $S_y(m,n)$ are linear homogeneous polynomials with rational coefficients (which depend only on $g$ but not $p,q$).
Plugging this into the second equation of \eqref{SysSplit}, we get the following quartic equations with integer coefficients w.r.t. $m,n$:
\beq\label{EP2S}
\ell_g\cdot P_2(S_x(m,n),S_y(m,n)) = q^2\cdot \frac{g\cdot \ell_g}{p^2}.
\eeq
where $\ell_g$ is the least common multiple of the coefficients denominators of $P_2(S_x(m,n),S_y(m,n))$ (notice that $\ell_g$ depends only on $g$ but not $p,q$).
Here $p^2$ must divide $g\ell_g$ and thus there is a finite number of such equations. 
By \cite[Theorem~3]{Alekseyev2011}, each such equation has a finite number of solutions,
unless $P_2(S_x(m,n),S_y(m,n)) = c\cdot T(m,n)^2$ for some polynomial $T(x,y)$, which is not the case since $\Disc(P_2)\ne 0$.
\end{proof}

\begin{remark} Different choices of values for $g,p,q$ may result in the same equation \eqref{EP2S}.
In particular, if $g'=g\cdot d^2$ for some integer $d$, then we can represent equation \eqref{EqQ1Q2} for $g'$ 
in the form $a\cdot Q_1(x,y)^2 + b\cdot Q_2(x,y)^2 - g\cdot (d\cdot z)^2 = 0$ so that it
has the same solutions w.r.t. $Q_1(x,y)$ and $Q_2(x,y)$. Then the equation \eqref{EP2S} for $g'$ has the same left hand side as the one for $g$ 
(with $\ell_{g'}=\ell_g$), while the former has an extra factor $d^2$ in the right hand side.
Therefore, to reduce the number of equations in practice, we can restrict $g$ to the square-free divisors of $G$: for each such $g$, we 
compute the left hand side of \eqref{EP2S} and iterate over all \emph{distinct} integer right hand sides of the form 
$q^2\cdot \tfrac{g\cdot \ell_g\cdot d^2}{p^2}$, where $d^2$ divides $\tfrac{G}{g}$.
\end{remark}

\section{Finding integral points on biquadratic curves}\label{SecBiQ}

Now we are ready to prove our main result:

\begin{theorem}\label{TBiqCur}
Finding integral points on a biquadratic curve
\beq\label{EqBiqCur}
y^2 = a\cdot x^4 + b\cdot x^2 + c
\eeq
with integer coefficients $a,b,c$ and $ac(b^2 - 4ac)\ne 0$, reduces to solving a finite number of quartic Thue equations.
\end{theorem}

\begin{proof} Multiplying the equation \eqref{EqBiqCur} by $4c$, we can rewrite it as a linear combination of three squares 
with non-zero integer coefficients:\footnote{Alternatively, we can multiply \eqref{EqBiqCur} by $4a$ and obtain 
another linear combination of three squares: $(2ax^2+b)^2+(4ac-b^2)\cdot 1^2 - 4ay^2 = 0$, which has smaller coefficients and thus may be preferable when $c\gg a$.}
$$(b^2 - 4ac) (x^2)^2 + 4cy^2 - (bx^2 + 2c)^2 = 0.$$
Denoting $X=x^2$, $Y=y$, $Z=bx^2 + 2c$, $A=b^2 - 4ac$, $B=4c$, $C=-1$, we get a Diophantine equation:
\beq\label{EqABCXYZ}
A\cdot X^2 + B\cdot Y^2 + C\cdot Z^2 = 0.
\eeq
If this equation is solvable with a particular solution $(X,Y,Z) = (X_0,Y_0,Z_0)$, $Z_0\ne 0$, 
then by Theorem~\ref{ThABC} its general solution is given by
$$(X,Y,Z) = r\cdot \left( P_x(m,n),\ P_y(m,n),\ P_z(m,n) \right),$$
where $P_i(m,n)$ are polynomials defined by \eqref{Pxyz}, 
$m,n$ are coprime integers, and $r$ is a rational number.
In our case, this solution should additionally satisfy the relation:
$$2c = Z - bX = r\cdot \left( P_z(m,n) - b\cdot P_x(m,n) \right),$$
implying that
$$r = \frac{2c}{ P_z(m,n) - b\cdot P_x(m,n) }.$$
So we get a constraining Diophantine equation:
\beq\label{EqXfrac}
x^2 = \frac{2c\cdot P_x(m,n)}{ P_z(m,n) - b\cdot P_x(m,n) },
\eeq
which reduces to a finite number of Thue equations by Theorem~\ref{TFracSq}, unless $\Res(P_x,P_z-b\cdot P_x)=0$ or $\Disc(P_z-b\cdot P_x)=0$.

The case of $\Res(P_x,P_z-b\cdot P_x)=\Res(P_x,P_z)=0$ is impossible since by direct computation we have $\Res(P_x,P_z) = -4AB^2CZ_0^4 \ne 0$.

In the case of $\Disc(P_z-b\cdot P_x)=0$,
we have
$\Disc(P_z-b\cdot P_x) = 4b^2B^2Y_0^2 - 4 AB (Z_0^2 - b^2 X_0^2) = 0$
and thus $b^2BY_0^2 = A(Z_0^2 - b^2 X_0^2)$.
Since $BY_0^2 = -A X_0^2 - C Z_0^2$, we further get $b^2(-A X_0^2 - C Z_0^2) = A(Z_0^2 - b^2 X_0^2)$, which reduces to $A+b^2C=0$. 
However this is impossible since $A+b^2C = -4ac \ne 0$.

Therefore, Theorem~\ref{TFracSq} is applicable.

\end{proof}

As a corollary we get

\begin{theorem}\label{Tzz2}
The system of Diophantine equations:
$$\begin{cases}
z = ax^2 + d_1, \\
z^2 = by^2 + d_2,
\end{cases}
$$
where $a,b,d_1,d_2$ are rational numbers with $abd_2(d_1^2-d_2)\ne 0$,
reduces to a finite number of quartic Thue equations.
\end{theorem}

\begin{proof} The system implies that $(ax^2 + d_1)^2 = by^2 + d_2$ or $(by)^2 = a^2bx^4 + 2abd_1x^2 + b(d_1^2 - d_2)$. Since
$(2d_1ab)^2 - 4a^2b^2(d_1^2 - d_2) = 4a^2b^2d_2\ne 0$, Theorem~\ref{TBiqCur} applies.
\end{proof}

\subsection*{Example: Ljunggren equation}

We illustrate our method on Ljunggren equation $y^2 = 2x^4 - 1$ whose reduction to Thue equations was first obtained in~\cite{Steiner1991}.
Here we have $(a,b,c)=(2,0,-1)$.
First we compute $A=b^2-4ac=8$, $B=4c=-4$, and $C=-1$ and consider equation \eqref{EqABCXYZ}.
Its particular solution is $(1,1,2)$, which yields by Theorem~\ref{ThABC} a general solution:
$$
(X, Y, Z) = \frac{p}{q}\; ( P_x(m, n),\; P_y(m, n),\; P_z(m, n))
$$
with
\begin{eqnarray*}
P_x(m,n) & = & 8 m^2 - 8mn + 4 n^2,\\
P_y(m,n) & = & -8 m^2 + 16 mn - 4 n^2,\\
P_z(m,n) & = & 16 m^2 - 8 n^2.
\end{eqnarray*}

Now we consider equation \eqref{EqXfrac}:
$$x^2 = \frac{2cP_x(m,n)}{P_z(m,n)-b\cdot P_x(m,n)} = \frac{-2(8 m^2 - 8mn + 4 n^2)}{16 m^2 - 8 n^2} = \frac{-2 m^2 + 2mn - n^2}{2 m^2 - n^2}$$
and use Theorem~\ref{TFracSq} to solve it. 
We take the resultant matrix of $P_1(x,y)=-2 x^2 + 2xy - y^2$ and $P_2(x,y)=2 x^2 - y^2$ and solve the two linear systems \eqref{EqTwoSys}
to obtain $(t_1,t_2,t_3,t_4)=\tfrac{1}{4}(-2, -1, 0, 1)$ and $(t_1,t_2,t_3,t_4)=(-1, -1, -1, 0)$. So we get $G=4$.
Let $g$ range over the divisors of $G$, i.e., $g\in\{-4,-2,-1,1,2,4\}$.

We use Theorem~\ref{TPol2Sq} to represent $P_1(x,y)$ as a linear combination of squares:
$$P_1(x,y)=-2 x^2 + 2xy - y^2 = - \frac{1}{8} \cdot (-4x + 2y)^2 - \frac{1}{2}\cdot y^2 = - \frac{1}{2} \cdot (-2x + y)^2 - \frac{1}{2}\cdot y^2$$
and obtain the equation \eqref{EqQ1Q2} (multiplied by $-2$): 
$$(-2x+y)^2 + y^2 + 2gz^2 = 0.$$
Clearly, it may have non-trivial solutions only for $g<0$ and only when $-1$ is a quadratic residue modulo $2g$, which leaves us the only suitable value $g=-1$.
The corresponding equation has a particular solution $(1,1,1)$ and by Theorem~\ref{ThABC} its general solution is
$$(-2x+y,y,z) = \frac{p}{q}\cdot\left( m^2 + 2mn - n^2, -m^2 + 2mn + n^2, m^2 + n^2 \right)$$
where $(p,q)=1$ and $q>0$ divides 4.

From this solution we express $x = \tfrac{p}{q}\cdot(-m^2 + n^2)$ and $y=\tfrac{p}{q}\cdot(-m^2 + 2mn + n^2)$ and plug them into the equation $P_2(x,y)=g$
to obtain the following quartic equations:
$$m^4 + 4m^3n - 6m^2n^2 - 4mn^3 + n^4 = q^2\cdot \frac{-1}{p^2}$$
and conclude that $p^2=1$. Since the polynomial in the left hand side is irreducible, these are Thue equations.

We used PARI/GP to solve the resulting three Thue equations (for $q=1,2,4$) and found that only the one with $q=2$ has solutions,
which are $(m,n) = \pm(5,-1)$, $\pm(1,5)$, $(\pm 1,\pm 1)$ (the same solutions were earlier obtained by Lettl and Peth\H{o} \cite{LettlPetho}).
The corresponding solutions to $\tfrac{P_1(m,n)}{P_2(m,n)}=x^2$ are 
$(m,n) = \pm(12,17)$ and $\pm(0,1)$, giving respectively $x=\pm 13$ and $\pm 1$. 
Thus the solutions to Ljunggren equation are $(\pm 13,\pm 239)$ and $(\pm 1,\pm 1)$.

\section{Near-multiples of squares in Lucas sequences}\label{SecInLucas}

The pair of Lucas sequences $U(P,Q)$ and $V(P,Q)$ are defined by the same linear recurrent relation with the coefficient $P,Q\in\mathbb{Z}$ but
different initial terms:
$$
\begin{array}{lll} 
U_0(P,Q)=0, & U_1(P,Q)=1, & U_{n+1}(P,Q) = P\cdot U_n(P,Q) - Q\cdot U_{n-1}(P,Q),\, n\geq 1;\\
V_0(P,Q)=2, & V_1(P,Q)=P, & V_{n+1}(P,Q) = P\cdot V_n(P,Q) - Q\cdot V_{n-1}(P,Q),\, n\geq 1.
\end{array}
$$

Some Lucas sequences have their own names:
\begin{center} \begin{tabular}{|l|l|l|l|}
\hline
Sequence & Name & Initial terms & Index in \cite{OEIS}\\
\hline\hline
$U(1,-1)$ & Fibonacci numbers & $0, 1, 1, 2, 3, 5, 8, 13, 21, 34, 55, 89, \ldots$ & \seqnum{A000045} \\
\hline
$V(1,-1)$ & Lucas numbers & $2, 1, 3, 4, 7, 11, 18, 29, 47, 76, 123, \ldots$ & \seqnum{A000032} \\
\hline
$U(2,-1)$ & Pell numbers &  $0, 1, 2, 5, 12, 29, 70, 169, 408, 985, \ldots$  & \seqnum{A000129} \\
\hline
$V(2,-1)$ & Pell-Lucas numbers &  $2, 2, 6, 14, 34, 82, 198, 478, 1154, \ldots$ & \seqnum{A002203} \\
\hline
\end{tabular}
\end{center}
Other examples include Jacobsthal numbers $U(1,-2)$, Mersenne numbers $U(3,2)$ etc.

The characteristic polynomial of Lucas sequences $U(P,Q)$ and $V(P,Q)$ is $\lambda^2 - P\lambda + Q$ with the discriminant $D=P^2 - 4Q$.
For non-degenerate sequences, the discriminant $D$ is a positive non-square integer.

Terms of Lucas sequences satisfy the following identity:
\begin{equation}\label{MainEq}
V_n(P,Q)^2 - D\cdot U_n(P,Q)^2 = 4 Q^n.
\end{equation}

In the current paper, we focus on the case of $Q=1$ or $Q=-1$, implying that
the pairs $(V_n(P,Q),U_n(P,Q))$ satisfy the equation:\footnote{Here and everywhere below $\pm$ in the r.h.s. of an equation means that we accept both signs as solutions.} 
\begin{equation}\label{EqXDY}
x^2 - D y^2 = \pm 4.
\end{equation}
The converse statement can be used to prove that given positive integers belong to $V(P,Q)$ or $U(P,Q)$ respectively:

\begin{theorem}[Theorem~1 in \cite{Alekseyev2011}]\label{Trecog} 
Let $P$, $Q$ be integers such that $P>0$, $|Q|=1$, $(P,Q)\ne(3,1)$, 
and $D=P^2 - 4Q > 0$.
If positive integers $u$ and $v$ are such that 
$$v^2 - D u^2 = \pm 4,$$
then
$u = U_n(P,Q)$ and $v = V_n(P,Q)$ for some integer $n\geq 0$.
\end{theorem}

\begin{theorem}\label{finit1} For fixed integers $a\ne 0$ and $b$, finding terms of the form $am^2 + b$ 
in nondegenerate Lucas sequence $U(P,Q)$ or $V(P,Q)$ with $Q=\pm 1$ reduces to a finite number of Thue equations, 
unless this sequence is $V(P,Q)$ and $b=\pm 2$. 
\end{theorem}

\begin{proof} Let $ax^2 + b$ be a term of $U(P,\pm 1)$. By Theorem~\ref{Trecog} we have $y^2 = D (ax^2+b)^2 \pm 4 = Da^2x^4 + 2Dabx^2 + (Db^2 \pm 4)$.
Theorem~\ref{TBiqCur} applies, since $Db^2 \pm 4\ne 0$ (notice that $D\ne\pm 1$) and $(2Dab)^2 - 4 Da^2(Db^2\pm 4) = \mp 16 Da^2\ne 0$.

Now let $ax^2 + b$ be a term of $V(P,\pm 1)$. By Theorem~\ref{Trecog} we have $(Dy)^2 = D((ax^2+b)^2 \mp 4) = Da^2x^4 + 2Dabx^2 + D(b^2 \mp 4)$.
We have $(2Dab)^2 - 4D^2a^2(b^2 \mp 4) = \pm 16D^2a^2\ne 0$. Theorem~\ref{TBiqCur} applies here as soon as $b^2 \mp 4\ne 0$ (i.e., $b\ne \pm 2$).
\end{proof}

\begin{remark} For $b=\pm 2$, the sequence $V(P,Q)$ may have an infinite number of terms of the form $am^2+b$.
In particular, Lucas sequence $V(1,-1)$ (Lucas numbers) has infinitely many terms of the forms $m^2+2$ and $m^2-2$ since
$V_{4n}(1,-1) = V_{2n}(1,-1)^2 - 2$ and $V_{4n+2}(1,-1) = V_{2n+1}(1,-1)^2 + 2$.
\end{remark}

The following theorem allows to find all terms solutions of the form $am^2 \pm 2$ in $V(P,\pm 1)$.

\begin{theorem}\label{finit2} For fixed integers $a\ne 0$ and $b=\pm 2$, finding terms of the form $am^2 + b$ 
in nondegenerate Lucas sequence $V(P,Q)$ with $Q=\pm 1$ reduces to a finite number of Thue equations and a Pell-Fermat equation.
\end{theorem}

\begin{proof}
Let $ax^2 + b$ with $b=\pm 2$ be a term of $V(P,\pm 1)$. By Theorem~\ref{Trecog}, we have the following equations:
$$(Dy)^2 = D((ax^2+b)^2 + 4) = Da^2x^4 + 2Dabx^2 + 8D$$
and 
$$(Dy)^2 = D((ax^2+b)^2 - 4) = Da^2x^4 + 2Dabx^2.$$
The former equation is addressed by Theorem~\ref{TBiqCur}, while the latter equation always has solution $x=0$ (corresponding to the term $V_0=2$) and
for $x\ne 0$ is equivalent to the Pell-Fermat equation:
$$\left(\tfrac{Dy}{x}\right)^2 - Da^2x^2 = 2Dab.$$
For solution of this equation, we refer to \cite[Section~6.3.5]{Cohen07}.
\end{proof}

Theorem~\ref{finit1} implies finiteness of the terms of the form $am^2+b$ in a Lucas sequence $U(P,Q)$ or $V(P,Q)$ with $Q=\pm 1$, unless 
this sequence is $V(P,Q)$ and $b=\pm 2$. The latter case is addressed by Theorem~\ref{finit2} and may sometimes result in infinitely many terms.
While this characterization represents a special case of the results by Nemes and Peth\H{o} \cite{NemesPetho} and by Peth\H{o} \cite{PethoGn},
our proofs rely on the simple transformation method developed in this paper.

In Table~\ref{Tab1} we list all the terms of the form $am^2 + b$ for $1\leq a\leq 3$ and $-3\leq b\leq 3$ in Fibonacci, Lucas, Pell, and Pell-Lucas numbers.

\begin{footnotesize}
\begin{table}
\begin{center}
\begin{tabular}{|l||l|l|l|l|}
\hline
     & Fibonacci numbers & Lucas numbers & Pell numbers & Pell-Lucas numbers \\
Form & $U(1,-1)$ & $V(1,-1)$ & $U(2,-1)$ & $V(2,-1)$ \\
\hline\hline
$m^2$ & $0, 1, 144$ & $1,4$ & $0, 1, 169$ & none \\
\hline
$m^2 + 1$ & $1, 2, 5$ & $1, 2$ & $1,2,5$ & $2,82$ \\
\hline
$m^2 - 1$ & $0, 3, 8$ & $3$ & $0$ & none \\
\hline
$m^2 + 2$ & $2,3$ & $2, 11$, and $V_{4n+2}$ & $2$ & $2$ and $V_{4n+2}$ \\
\hline
$m^2 - 2$ & $2, 34$ & $V_{4n}$ & $2$ & $14$ and $V_{4n}$ \\
\hline
$m^2 + 3$ & $3$ & $3, 4, 7, 199$ & $12$ & none \\
\hline
$m^2 - 3$ & $1, 13, 1597$ & $1$ & $1$ & $6$ \\
\hline
$2m^2$  & $0, 2, 8$ & $2, 18$ & $0, 2$ & $2$ \\
\hline
$2m^2 + 1$ & $1, 3$ & $1, 3$ & $1$ & none \\
\hline
$2m^2 - 1$ & $1$ & $1, 7, 199$ & $1$ & none \\
\hline
$2m^2 + 2$ & $2, 34$ & $2,4$ & $2$ & $2$ and $V_{4n}$ \\
\hline
$2m^2 - 2$ & $0$ & none & $0, 70$ & $V_{4n+2}$ \\

\hline
$2m^2 + 3$ & $3,5,21$ & $3, 11$ & $5$ & none \\
\hline
$2m^2 - 3$ & $5$ & $29, 47, 64079$ & $5, 29$ & none \\

\hline
$3m^2$  & $0, 3$ & $3$ & $0, 12$ & none \\
\hline
$3m^2 + 1$ & $1, 13$ & $1, 4, 76$ & $1$ & none \\
\hline
$3m^2 - 1$ & $2$ & $2, 11, 47$ & $2$ & $2$ \\
\hline
$3m^2 + 2$ & $2, 5$ & $2, 29$ & $2, 5, 29$ & $2,14$ \\
\hline
$3m^2 - 2$ & $1$ & $1$ & $1$ &  none \\
\hline
$3m^2 + 3$ & $3$ & $3$ & none & $6$ \\
\hline
$3m^2 - 3$ & $0,144$ & none & $0$ & none \\
\hline
\end{tabular}
\caption{Terms of the form $am^2 + b$ for $1\leq a\leq 3$ and $-3\leq b\leq 3$ in Fibonacci, Lucas, Pell, and Pell-Lucas numbers.\label{Tab1}}
\end{center}
\end{table}
\end{footnotesize}

\section*{Acknowledgements}

We thank Dr. John Cremona, Dr. Nikos Tzanakis, and Dr. Noam Elkies for a number of helpful suggestions and discussions.

The first author was supported by the National Science Foundation under Grant No. IIS-1253614. The second author was partially 
supported by the European Union and the European Social Fund through project ``Supercomputer, the national virtual lab"
under Grant No. T\'AMOP-4.2.2.C-11/1/KONV-2012-0010.

\bibliographystyle{plain} 
\bibliography{near-new.bib} 

\bigskip
\hrule
\bigskip

\noindent 2010 {\it Mathematics Subject Classification}:
Primary 11Y50; Secondary 11D25, 11B39, 14G05. 

\noindent \emph{Keywords: }
integral point, biquadratic curve, elliptic curve, Thue equation, Fibonacci number, Lucas sequence

\bigskip
\hrule
\bigskip

\noindent (Concerned with sequences
\seqnum{A000032},
\seqnum{A000045},
\seqnum{A000129},
\seqnum{A002203}.)

\end{document}